\documentclass{amsart}
\usepackage[utf8]{inputenc}
\usepackage{setspace}
\usepackage[margin=1.25in]{geometry}
\usepackage{graphicx}
\graphicspath{ {./figures/} }
\usepackage{subcaption}
\usepackage{amsmath}
\usepackage{amssymb}
\usepackage{mathtools}
\usepackage{lineno}
\usepackage{amsfonts}
\usepackage{xfrac} 
\usepackage{graphicx} 
\usepackage{amsthm}
\usepackage[all]{xy}

\usepackage{float}
\usepackage{braket}

\usepackage{kbordermatrix}

\usepackage{multirow}
\usepackage{diagbox}
\usepackage{slashbox}

\usepackage{tikz-cd}
\usetikzlibrary{cd}

\usepackage[pagebackref=true]{hyperref}

\newtheorem{theorem}{Theorem}

\newtheorem{definition}[theorem]{Definition}
\newtheorem{lemma}[theorem]{Lemma}
\newtheorem{corollary}[theorem]{Corollary}

\newtheorem*{theorem*}{Theorem}

\title{Topology of the orbit space of a maximal torus action on  $\operatorname{Gr}(2,\mathbb{C}^{5})$}

\author{Shahryar Ghaed Sharaf}

\begin{document}
	
	\maketitle

	\begin{abstract}
In this work, we take a topological approach and introduce a homotopy model of the orbit space of a maximal torus action on $\operatorname{Gr}(2, \mathbb{C}^{5})$. Using this homotopy model, we compute the integral homology groups and study the homotopy type of the orbit space. More precisely, we endow the homotopy model with a CW complex that contains one top cell and determine the homotopy class of the attaching map of the top cell of the orbit space.
	\end{abstract}

\tableofcontents

\section{Introduction}

Consider the Grassmannian $\operatorname{Gr}(2, \mathbb{C}^{5})$ parametrising planes in $\mathbb{C}^{5}$. The action of the algebraic torus $(\mathbb{C}^{\ast})^{5} \curvearrowright \mathbb{C}^{5}$ induces an action of the subtorus $T_{\mathbb{C}} = \{ (t_{1}, \dots, t_{5}) \in (\mathbb{C}^{\ast})^{5} \mid t_{1} \dots t_{5} = 1 \}$ on $\operatorname{Gr}(2, \mathbb{C}^{5})$. In this paper, we study the homotopy type of the orbit space $\sfrac{\operatorname{Gr}(2, \mathbb{C}^{5})}{T}$, where $T = \{ (t_{1}, \dots, t_{5}) \in (S^{1})^{5} \mid t_{1} \dots t_{5} = 1 \}$ is the associated compact torus. The homology groups of the orbit space are well-known and obtained by different methods. Süss computed the homology of the orbit space in \cite{suss2020toric} using a filtration of subspaces $\sfrac{\operatorname{Gr}(2, \mathbb{C}^{5})}{T}$. Buchstaber and Terzić determined the homology groups of the orbit space using parameter spaces in \cite{bukhshtaber2019toric}.

In this work, we take a different approach to study the topology of the orbit space $\sfrac{\operatorname{Gr}(2, \mathbb{C}^{5})}{T}$ and compute its homology groups. We introduce a homotopy model of the orbit space which is relatively easy to study. For our purposes, considering the contraction maps as algebraic morphisms is too restrictive. More precisely, we start by studying the topological properties of the five conic bundles from the degree 5 Del Pezzo surface to $\mathbb{CP}^{1}$ considered as maps between topological manifolds. Then we use these conic bundles to construct the homotopy model. In this work, we use the algebraic terminology and denote 4-dimensional topological manifolds by surfaces and 2-dimensional topological submanifolds by curves. Our homotopy model is constructed as follows. The underlying moment polytope of the torus action is the hypersimplex $\Delta_{2,5}$, which is homeomorphic to a 4-dimensional disc, $D^4$. The hypersimplex has $10$ vertices, $30$ edges, $30$ $2$-faces ($20$ triangles and $10$ squares), and $10$ 3-faces ($5$ tetrahedra and $5$ octahedra). We consider the interior of the hypersimplex $\Delta_{2,5}$ as a 4-dimensional cell. The above decomposition of the topological boundary of the underlying polytope yields a regular CW complex on the disc $D^4$. Note that a CW complex is called regular if its characteristic maps can be chosen to be embeddings. Let $Y=dP_{5}$ be the degree 5 del Pezzo surface. We have the homeomorphism $Y \cong \mathbb{CP}^{2} \#  \overline{\mathbb{CP}^{2}}\#  \overline{\mathbb{CP}^{2}}\#  \overline{\mathbb{CP}^{2}}\#  \overline{\mathbb{CP}^{2}}$. First, to each point in the interior of the disc $D^4$, we attach the topological space $Y$. It is possible to choose five sets of pairwise disjoint (-1)-curves in $Y$. Contracting each set of (-1)-curves yields an algebraic morphism $Y \longrightarrow \mathbb{CP}^{1}$. We first study the topology of these algebraic morphisms. Using the morphisms $Y \longrightarrow  \mathbb{CP}^{1}$, we attach copies of $\mathbb{CP}^{1}$ to each point in the interior of octahedra in the topological boundary of the hypersimplex $\Delta_{2,5}$. In the next step, we show that the resulting space, denoted by $X$, is homotopy equivalent to the orbit space $\sfrac{\operatorname{Gr}(2, \mathbb{C}^{5})}{T}$. We can endow the topological space $X$ with a CW complex obtained from the above regular CW complex of the disc $D^{4}$ and minimal CW complexes on $\mathbb{CP}^{2} \# \overline{\mathbb{CP}^{2}} \#  \overline{\mathbb{CP}^{2}}\#  \overline{\mathbb{CP}^{2}}\#  \overline{\mathbb{CP}^{2}}$ and $\mathbb{CP}^{1}$. Computing the homology groups of the associated chain complex the associated chain complex of the obtained CW complex on the topological space $X$ is straightforward. In the final section, we study the homotopy properties of the orbit space $\sfrac{\operatorname{Gr}(2, \mathbb{C}^{5})}{T}$, using the topological space $X$. Theorems~\ref{Theo1} and~\ref{Theo2} are the main results of this paper and can be summarized as follows:

 \begin{theorem*}
		The orbit space satisfies $\sfrac{\operatorname{Gr}(2, \mathbb{C}^{5})}{T} \simeq D^{8} \bigcup_{f} \Sigma^{4} \mathbb{RP}^{2}$, where the homotopy class $[f] \in \pi_{7}(\Sigma^{4} \mathbb{RP}^{2} )$ is a generator. Moreover, we have $D^{8} \bigcup_{f} \Sigma^{4} \mathbb{RP}^{2} \not \simeq \Sigma^{4}(D^{4} \bigcup_{g} \mathbb{RP}^{2}) $ for all $[g] \in \pi_{3}(\mathbb{RP}^{2})$.
\end{theorem*}

The novelty of this work is twofold. First, it gives a substantially simpler computation of the homology groups of the orbit space than the one in Buchstaber and Terzić \cite{bukhshtaber2019toric}. Second, it determines the homotopy type of the orbit space — a result not obtained by Süss \cite{suss2020toric}.

\section{Topology of the Contraction Maps}\label{TopConst}

      The main goal of this section is to study the topology of the contraction maps. We consider these maps as topological quotient maps and use the minimal CW complex of the blow-ups of $\mathbb{CP}^{2}$ to compute the induced homomorphisms on the homology groups.

      Note that, to obtain the Del Pezzo surface $Y = dP_5$ of degree 5, one can blow up $\mathbb{CP}^{2}$ at 4 points in general position. As a result, the variety $Y$, considered as a topological space, becomes homeomorphic to a connected sum of $\mathbb{CP}^{2}$, which we refer to as the original one, along with four other copies of $\mathbb{CP}^{2}$ with reversed orientation. The minimal CW complex of the connected sum $Y=\mathbb{CP}^{2} \#  \overline{\mathbb{CP}^{2}}\#  \overline{\mathbb{CP}^{2}}\#  \overline{\mathbb{CP}^{2}}\#  \overline{\mathbb{CP}^{2}}$ is given by 
      \begin{align}\label{delPezzo}
    	Y = e^{0}_{Y} \cup \big( e^{2}_{Y_{5}} \cup ( \bigcup_{i=1}^{4} e^{2}_{Y_{i}}) \big) \cup e^{4}_{Y}.
      \end{align}
      Note that in this case, each generator of the homology group $H_{i}(Y)$ corresponds to an $i$-cell in the minimal CW complex on $Y$. For our purposes, we need to study the homotopy class of the attaching map of the top cell in $Y$. The intersection form (cup product) $H^{2}(Y; \mathbb{Z}) \times H^{2}(Y; \mathbb{Z}) \xrightarrow{\cup} H^{4}(Y; \mathbb{Z})$ can be given by the matrix $\operatorname{diag}(-1,-1,-1,-1,1)$. As shown in \cite{freedman1982topology} by Freedman, two closed, simply connected topological 4‑manifolds are homeomorphic if and only if they have isomorphic intersection forms and the same Kirby–Siebenmann invariant. Note that $\mathbb{CP}^{2}$ is smooth and hence has a vanishing Kirby–Siebenmann invariant. Furthermore, note that $\pi_{3}(\bigvee_{l=1}^{5} S^{2}_{l}) \cong \bigoplus_{l=1}^{5}\mathbb{Z} [\eta_{l}] \bigoplus_{1 \leq j < k \leq 5}\mathbb{Z} [i_{j},i_{k}]\cong \mathbb{Z}^{5} \oplus \mathbb{Z}^{10}$, where $\eta_{n}:S^{3} \longrightarrow \bigvee_{l=1}^{5} S^{2}_{l}$ is the Hopf map into the $n$-th sphere factor, and $[i_{j},i_{k}]$ denotes the Whitehead product of the inclusions $i_{j}:S^{2}_{j} \hookrightarrow \bigvee_{l=1}^{5} S^{2}_{l}$ and $i_{k}:S^{2}_{k} \hookrightarrow \bigvee_{l=1}^{5} S^{2}_{l}$. For details, see Hilton \cite{hilton1955homotopy}. Let $f:S^{3} \longrightarrow \bigvee_{l=1}^{5} S^{2}_{l}$ be the attaching map of the top cell of the topological space $Y$. Since the intersection form of $Y$ is diagonal, there is no contribution from the Whitehead products to the homotopy class $[f] \in \pi_{3}(\bigvee_{l=1}^{5} S^{2}_{l})$. From the intersection form we can read off the homotopy class of the map $f$, and it follows that $[f] = \sum_{l=1}^{4}(-[\eta_{l}])+[\eta_{5}]$. Using the above notation, it is also worthwhile to mention that the intersection form of $\mathbb{CP}^{1} \times \mathbb{CP}^{1}$ is given by $	\left( \begin{matrix}
	  0 & 1 \\ 1 & 0 \end{matrix} \right)$. Note that the second skeleton $(\mathbb{CP}^{1} \times \mathbb{CP}^{1})^{2} \cong S^{2}_{1} \vee S^{2}_{2}$. Hence, the homotopy class of the attaching map $g:S^{3} \longrightarrow S^{2}_{1} \vee S^{2}_{2}$ is given by $[g]=[i_{1},i_{2}]$, where $\pi_{3}(S^{2}_{1} \vee S^{2}_{2}) \cong \mathbb{Z}[\eta_{1}]\oplus \mathbb{Z}[\eta_{2}] \oplus \mathbb{Z}[i_{1},i_{2}]$.

       Let $E_{1}, \dots , E_{4}$ be the exceptional divisors of the blow-ups, and let $E_{5}$ be the preimage of a complex line in the original $\mathbb{CP}^{2}$ that does not contain the points of the blow-up. The Picard group $\operatorname{Pic}(Y)$ is a free abelian group of rank 5 with a basis given by $[E_{1}], \dots , [E_{5}]$. The intersection product of curves defines a non-degenerate bilinear form on $\operatorname{Pic}(Y)$, given by:
       \begin{align*}
	   [E_{i}] \cdot [E_{j}] = \left\{ \begin{array}{rc}
	 	1 & i=j=5\\ -1 & i=j \leq 4 \\
		0 &  i \neq j. 
    	\end{array}\right.
      \end{align*}
       The classes of ten irreducible rational complex (-1)-curves are $[C_{i5}]=[E_{i}] \; , \; i \in \{1, \dots , 4\}$, and $[C_{ij}]=[E_{5}]-[E_{i}]-[E_{j}] \; , \;\{i,j\} \subset \{1, \dots, 4\}$. Note that $[C_{ij}]$ and $[C_{lk}]$ are disjoint if and only if $[C_{ij}] \cdot [C_{kl}] = 0$. Hence, we can deduce the following rules for the intersection of the ten irreducible curves. The curves representing the classes $[C_{i5}]$ and $[C_{j5}]$ are disjoint if and only if $i \neq j$. For $\{i,j\}, \{k,l\} \subset \{1,2,3,4\} $, it follows that the curves representing $[C_{ij}]$ and $[C_{kl}]$ are disjoint if and only if $\vert \{i,j\}\cap \{k,l\} \vert=1$. Finally, the curves representing the classes $[C_{i5}]$ and $[C_{kl}]$ intersect if $i \in \{k,l\}$, where $\{j,k\} \subset \{1,2,3,4\}$. Let $D \subset \big\{ [C_{ij}] \; \vert \; \{i,j  \} \subset \{ 1,2,3,4,5 \}  \big\}$ and $\operatorname{card}(D) \leq 4$, such that the elements of $D$ are pairwise disjoint. Contracting the complex curves in $D$ yields a birational model $Y_{D}$ obtained from $Y$. The obtained birational model, $Y_D$, is homeomorphic to either the blow-up of $\mathbb{CP}^{2}$ on up to 3 points in general positions or $\mathbb{CP}^{1} \times \mathbb{CP}^{1}$. Using the minimal CW complex on $Y$, introduced in Equation \ref{delPezzo}, the contraction maps $\phi_{ \{[C_{i5}]\}}:Y \longrightarrow Y_{ \{[C_{i5}]\} }$ are obtained by collapsing the 2-cell $e^{2}_{i}$ to a point where $i \in \{1,2,3,4\}$. Note that the first Chern homomorphism $c_{1}:\operatorname{Pic}(Y) \longrightarrow H^{2}(Y)$ is an isomorphism and by using Poincaré duality and abuse of notation, we conclude that the set $\{[E_{1}],[E_{2}],[E_{3}],[E_{4}],[E_{5}]\}$ is a basis of $H_{2}(Y)$. The above non-degenerate bilinear form on $\operatorname{Pic}(Y)$ yields the intersection form $Q_{Y}$ on $H_{2}(Y)$, namely $Q_{Y}=\operatorname{diag}(-1,-1,-1,-1,1)$. We should also mention that for a closed oriented 4-manifold the intersection form (cup product) on the second cohomology group $H^{2}$ and the intersection form on the second homology group $H_{2}$ are the same bilinear form using Poincaré duality. As a result, the induced homomorphism on the second homology group is the projection 
       $(\phi_{ \{[C_{i5}]\}})_{2}: H_{2}(Y) (\cong \bigoplus_{l=1}^{5}\mathbb{Z} \braket{[E_{l}]}) \longrightarrow H_{2}(Y_{ \{[C_{i5}]\}})(\cong \sfrac{(\bigoplus_{l=1}^{5} \mathbb{Z}\braket{[E_{l}]})}{\mathbb{Z}\braket{[E_{i}] }} )$. The description of the contraction maps $\phi_{ \{[C_{ij}]\}}:Y \longrightarrow Y_{ \{[C_{ij}]\} }$ is a bit more involved. Contracting the (-1)-curve $[C_{ij}]$ induces the homomorphism $(\phi_{ \{[C_{ij}]\}})_{2}: H_{2}(Y) (\cong \bigoplus_{l=1}^{5}\mathbb{Z} \braket{[E_{l}]}) \longrightarrow H_{2}(Y_{ \{[C_{ij}]\}})\big(\cong \sfrac{(\bigoplus_{l=1}^{5} \mathbb{Z}\braket{[E_{l}]})}{\big( \mathbb{Z}(\braket{[E_{5}]} -\braket{ [E_{i}]}-\braket{[E_{j}]} )\big) } \big)$ between the second homology groups. Hence, we can choose a basis of $H_{2}(Y_{ \{[C_{ij}]\}})$ such that 
       \begin{align*}
	   (\phi_{ \{[C_{ij}]\}})_{2}=	\left( \begin{matrix}
		1 & 0 & 0 & 0 & 1  \\
		0 & 1 & 0 & 0 & 1  \\
		0 & 0 & 1 & 0 & 0  \\
		0 & 0 & 0 & 1 & 0  \\		
	   \end{matrix} \right).
       \end{align*}
       Recall that each 2-cell in the minimal CW complex on $Y_{ \{[C_{ij}]\}}$ is represented by a generator of the second homology group $H_{2}(Y_{ \{[C_{ij}]\}})$. Consider the minimal CW complex $Y_{ \{[C_{ij}]\}} = e^{0}_{Y_{ \{[C_{ij}]\}} } \cup ( \bigcup_{i=1}^{4} e^{2}_{{Y_{ \{[C_{ij}]\}}}_{i}} ) \cup e^{4}_{Y_{ \{[C_{ij}]\}}}$. From the matrix representation of $(\phi_{ \{[C_{ij}]\}})_{2}$ it follows that the contraction map $\phi_{[C_{ij}]}$ sends each 2-cell $e^2_{Y_l}$ (for $l=1,2,3,4$) to the cell $e^2_{{Y_{\{[C_{ij}]\} }}_{l} }$, and sends $e^2_{Y_5}$ to the union $e^2_{{Y_{\{ [C_{ij}] \}}}_{1} }$ and $e^2_{{Y_{\{[C_{ij}]\} }}_{2}}$, with degree one on each of these two target cells.

       Using the above description, we can topologically describe all contraction maps $\phi_{D}$ and their corresponding induced homomorphisms on the homology groups.

       In the following, we briefly explain some topological properties of five distinct conic bundles $f_{l}:Y \longrightarrow \mathbb{CP}^{1}$ for $l \in \{1, \cdots ,5\}$. For a general introduction and the definitions of terms pencils and proper transforms the reader may consult the standard text book \cite[Chapter 1]{griffiths2014principles} by Griffiths and Harris. Let $p_{1}$, $p_{2}$, $p_{3}$, and $p_{4}$ denote the four blow-up points in general positions. Consider a pencil of lines through the point $p_{1}$. This pencil (family of lines) is parametrized by the algebraic variety $\mathbb{CP}^{1}$. For a general line $\widetilde{C}$ through $p_{1}$ (not passing through any other $p_{i}$), its proper transform $C$ on $Y$ is a smooth rational curve, homeomorphic to the topological manifold $\mathbb{CP}^{1}$, in the divisor class $[E_{5}]-[E_{1}]$. These proper transforms give a one-parameter family of disjoint curves each intersecting $E_{1}$ transversally in one point. If the line $\widetilde{C}_{12}$ passes through $p_{1}$ and $p_{2}$, its proper transform $C_{12}$ is a class in $[E_{5}]-[E_{1}]-[E_{2}]$. In other words, the class $[C_{12}]+[E_{2}]$ is equivalent to $[E_{5}]-[E_{1}]$. The same happens for the lines through $p_{1}$ and $p_{3}$, namely $\widetilde{C}_{13}$ with the proper transform $C_{13}$, and the line through $p_{1}$ and $p_{4}$, namely $\widetilde{C}_{14}$ with the proper transform $C_{14}$, respectively. We denote the first introduced pencil by $\vert E_{5}-E_{1} \vert$ and it defines a conic bundle $f_{1}:Y \longrightarrow \mathbb{CP}^{1}$ whose general fiber is a smooth rational curve and whose three special fibers are $C_{12} \cup E_{2}$, $C_{13} \cup E_{3}$, and $C_{14} \cup E_{4}$. The remaining four conic bundles $f_{2}$, $f_{3}$, $f_{4}$, and $f_{5}$ are constructed as above by considering the pencils $\vert E_{5}-E_{2} \vert$, $\vert E_{5}-E_{3} \vert$, $\vert E_{5}-E_{4} \vert$, and $\vert 2 E_{5}- \sum_{i=1}^{}4E_{i} \vert$, respectively. Using the above notation, the special fibers of the conic bundle $f_{5}$ are $C_{12} \cup C_{34}$, $C_{13} \cup C_{24}$, and $C_{14} \cup C_{23}$. For the purpose of this work, we need to study the induced homomorphisms $(f_{i})_{\ast}:H_{2}(Y) \longrightarrow H_{2}( \mathbb{CP}^{1})$. For each of the five conic bundles $f_{i}$, the kernel of $(f_{i})_{\ast}$ is a rank 4 sublattice of $H_{2}(Y)$. It is generated by the six $(-1)$-curves that appear as components of the three singular fibres. Let $\{[E_{1}],[E_{2}],[E_{3}],[E_{4}],[E_{5}]\}$ be a basis of $H_{2}(Y)$. Hence, we arrive at $\operatorname{ker}((f_{1})_{\ast})= \operatorname{span}\braket{ [E_{2}],[E_{3}],[E_{4}],[C_{12}],[C_{13}],[C_{14}]}$ and $\operatorname{ker}((f_{5})_{\ast})= \operatorname{span}\braket{[C_{12}],[C_{13}],[C_{14}],[C_{23}],[C_{24}],[C_{34}]}$. The kernels $\operatorname{ker}((f_{2})_{\ast})$, $\operatorname{ker}((f_{3})_{\ast})$, and $\operatorname{ker}((f_{4})_{\ast})$ are obtained similarly. Let $f_{\ast}= ((f_{1})_{\ast},(f_{2})_{\ast},(f_{3})_{\ast},(f_{4})_{\ast},(f_{5})_{\ast}) \in M_{5 \times 5}(\mathbb{Z})$. In the above basis, we have 
       \begin{align}\label{bound1}
        f_{\ast}=	\left( \begin{matrix}
		1 & 0 & 0 & 0 & 1  \\
		0 & 1 & 0 & 0 & 1  \\
		0 & 0 & 1 & 0 & 1  \\
		0 & 0 & 0 & 1 & 1  \\
		1 & 1 & 1 & 1 & 2  \\		
    	\end{matrix} \right).
       \end{align}

	\section{A homotopy Model of the Orbit Space}

	The hypersimplex $\Delta_{2,5}$ has 10 3-faces (facets), 30 2-faces, 30 1-faces (edges), and 10 0-faces (vertices). Consider the following regular CW complex on the hypersimplex $P:=\Delta_{2,5}$. We represent the interior of $P$ by a 4-cell. Each octahedral and tetrahedral facet are represented by $e^{3}_{O_{i}}$ and $e^{3}_{T_{i}}$, respectively. Similarly, we represent the remaining $n$-dimensional faces with $e^{n}_{P_{i}}$, where for $n=1,2$ $i \in \{1, \dots , 30\}$ and for $n=0$, $i \in \{0 , \dots , 10\}$. Let $f_{i}:Y \longrightarrow \mathbb{CP}^{1}$ for $i \in \{1, \cdots,5\}$ be the five distinct conic bundles described in the previous section. We define the topological space
	\begin{align}\label{HomModel}
	X = \sfrac{P \times Y}{\sim},
    \end{align}
	where $\sim$ is defined as follows. For $x, x^{\prime} \in \operatorname{int}(P)$ and $y,y^{\prime} \in Y$, we set $(x,y) \sim (x^{\prime},y^{\prime})$, if $(x,y) = (x^{\prime},y^{\prime})$. Only, if $x \in \operatorname{int}(O_{i})$, we require $(x,y) \sim (x,y^{\prime})$ if $f_{i}(y)=f_{i}(y^{\prime})$, in other words, we collapse $\{x\} \times Y$ to $\{x\} \times \mathbb{CP}_{i}^{1}$ using the conic bundle as the attaching map $f_{i}:Y \longrightarrow \mathbb{CP}^{1}$. Otherwise, we define $(x,y) \sim (x,\ast)$ for a fixed point $\ast \in Y$, which implies that over a point $x \in \partial(P) \setminus (\cup_{i} \operatorname{int}(\mathcal{O}_{i}))$ we collapse the topological manifold $Y$ to a point. We consider the given CW complex of $Y$ in Equation \ref{delPezzo} and the minimal CW complex on $\mathbb{CP}^{1}_{i}$. Recall that the generators of the homology groups $H_{i}(Y)$ can be considered as $i$-cells of the minimal CW complex of the topological manifold $Y$. Accordingly, we arrive at the following chain groups.
	   \begin{align} \label{CWX}
		C_{8}(X)&= \braket{ e^{4}_{Y} \times e^{4}_{P} } \mathbb{Z}, \; C_{7}(X)= 0, \; C_{6}(X) =\bigoplus_{i=1}^{5} \braket{e^{2}_{Y_{i}} \times e^{4}_{P}  } \mathbb{Z} \notag \\
		C_{5}(X) &= \bigoplus_{i=1}^{5} \braket{e^{2}_{\mathbb{CP}_{i}} \times e^{3}_{O_{i}} } \mathbb{Z}, \;
		C_{4}(X) = \braket{ e^{0}_{Y} \times e^{4}_{P} } \mathbb{Z} \\
        C_{3}(X) &= \bigoplus_{i=1}^{5} \braket{e^{0}_{\mathbb{CP}^{1}_{i}} \times e^{3}_{O_{i}}  } \mathbb{Z} \bigoplus_{i=1}^{5} \braket{e^{3}_{T_{i}} } \mathbb{Z} \notag \\
        C_{2}(X)&= \bigoplus_{i=1}^{30} \braket{e^{2}_{P_{i}}} \mathbb{Z}, \;
        C_{1}(X)= \bigoplus_{i=1}^{30} \braket{e^{1}_{P_{i}}} \mathbb{Z}, \;	 	
        C_{0}(X)= \bigoplus_{i=1}^{10} \braket{e^{0}_{P_{i}}} \mathbb{Z}.  \notag
    \end{align}
    By using the matrix representations of the induced homomorphisms $(f_{i})_{\ast}$, given in Equation \ref{bound1}, and considering that the 4-skeleton of the topological space $X$ is contractible, since it is homeomorphic to a 4-disk, we arrive at the following boundary operators
    \begin{align*}
	\partial_{i}&=0 \;\; \text{for}\;\; i \geq 5 \; \text{and} \; i \neq 6 \\
	\partial_{6}&= \left( \begin{matrix}
		1 & 0 & 0 & 0 & 1  \\
	0 & 1 & 0 & 0 & 1  \\
	0 & 0 & 1 & 0 & 1  \\
	0 & 0 & 0 & 1 & 1  \\
	1 & 1 & 1 & 1 & 2  \\		
	\end{matrix} \right) \\
    \partial_{i} &= \partial_{i}^{P} \; \;\text{for} \; \; i \leq 4,
    \end{align*}
     where $\partial_{i}^{P}$ denotes the boundary operator of the regular CW complex on $P$ and the boundary operator $\partial_{6}$ follows from the consideration that leads to Equation \ref{bound1}.

    \begin{corollary}\label{HomX}
	Let $X$ be constructed as above. Then, we have
	\begin{align*}
		H_{i}(X) = \left\{ \begin{array}{rcl}
		\mathbb{Z}_{\phantom{2}} & \text{for} & i=0,8 \\
		\mathbb{Z}_{2} & \text{for} & i=5\\ 0_{\phantom{2}} &  & \text{otherwise} \end{array} \right.
	\end{align*}
    \end{corollary}
    \begin{proof}
	Note that $\vert \operatorname{det}(\partial_{6})\vert=2$. The claim follows as a result.
    \end{proof}

    \begin{definition}
	Let $D$, $D'$, $W$, and $Z$ be topological spaces, and let $g:W \longrightarrow Z$ be a map. Let $D'$ and $D$ be closed in $D \cup D'$. The pushout $W \times D \cup_{g} Z \times D'$ is referred to as the space obtained by attaching $Z \times D'$ to $W \times D$ using the map $g:W \longrightarrow Z$ as the attaching map. If the inclusions $D' \cap D \hookrightarrow D$ and $D' \cap D \hookrightarrow D'$ are clear from the context, we omit the spaces $D'$ and $D$ and refer to the pushout as the space obtained by attaching $Z$ to $W$. 
    \end{definition}

      We now turn our attention to the topological construction of the orbit space. The construction proceeds as follows: there is a chamber decomposition of the hypersimplex $\Delta_{2,5}$, and each chamber is associated with a birational model. The chambers are attached together using the contraction maps. Our goal is to study the homotopy type of the orbit space. As it turns out, we do not require the exact arrangement of the chambers or the associated birational models. For a detailed construction of the chamber decomposition of the orbit space, the reader may consult \cite{suss2020toric} by Süss. For our topological considerations, we will need the following observations.

       Let $D_{i}$ and $D_{j}$ be sets of pairwise disjoint $(-1)$-curves such that $D_{j} \subset D_{i}$. Then, we have the following commutative diagram of the contraction maps.
       \begin{align}\label{CommDiag}
	   \xymatrix{Y \ar[d] \ar[dr] &  \\
		Y_{D_{i}} \ar[r] & Y_{D_{j}}} 
       \end{align}	

       Let $P_C$ be the central chamber and $P_{D_{i}}$ one of its neighboring chambers. Let $p: \sfrac{\operatorname{Gr}(2,\mathbb{C}^{5})}{T} \longrightarrow P$ be the natural projection, where $P \cong \Delta_{2,5}$. Let also $Y_{D_{i}}$ be the birational model associated with $P_{D_{i}}$. For each variety $Y_{D_{j}}$ on the topological boundary of $P_{D_{i}}$, we have a commutative diagram similar to Diagram \ref{CommDiag}. We attach $(Y \times D^{4})$ to $p^{-1}(\partial (P_{C} \cup P_{D_{i}}) )$ with the attaching maps $\phi_{j}:Y \longrightarrow Y_{D_{j}}$ for each birational model $Y_{D_{j}}$ associated with a cell in the regular CW complex of $\partial (P_{C} \cup P_{D_{i}})$ induced by the chamber decomposition of $P$. In other words, the pushout $(Y \times D^{4}) \cup p^{-1}(\partial (P_{C} \cup P_{D_{i}}) )$ denotes the space obtained by attaching $(Y \times D^{4})$ to $p^{-1}(\partial (P_{C} \cup P_{D_{i}}) )$ using the contraction maps $\phi_{j}$ as the attaching maps, as $j$ ranges over all cells in the regular CW complex $\partial (P_{C} \cup P_{D_{i}})$. Furthermore, we define the topological space $\big((\sfrac{\operatorname{Gr}(2, \mathbb{C}^{5})}{T})\setminus p^{-1}(\operatorname{int}(P_{C} \cup P_{D_{i} })) \big) \cup \big((Y \times D^{4}) \cup p^{-1}(\partial (P_{C} \cup P_{D_{i}}))\big)$ to be the space obtained by attaching $(Y \times D^{4}) \cup p^{-1}(\partial (P_{C} \cup P_{D_{i}}))$ to $(\sfrac{\operatorname{Gr}(2, \mathbb{C}^{5})}{T})\setminus p^{-1}(\operatorname{int}(P_{C} \cup P_{D_{i} }))$ by using the contraction maps $\phi_{j}$ as attaching maps.
      
      \begin{definition}\label{Replacing}
	   We refer to the topological space $\big((\sfrac{\operatorname{Gr}(2, \mathbb{C}^{5})}{T})\setminus p^{-1}(\operatorname{int}(P_{C} \cup P_{D_{i} })) \big) \cup \big((Y \times D^{4}) \cup p^{-1}(\partial (P_{C} \cup P_{D_{i}}))\big)$ as the space obtained from $\sfrac{\operatorname{Gr}(2, \mathbb{C}^{5})}{T}$ by replacing the chamber $P_{D_{i}}$ with a copy of the central chamber. By abuse of notation, we call the topological subspace $(Y \times D^{4}) \cup p^{-1}(\partial (P_{C} \cup P_{D_{i}})) \subset \big((\sfrac{\operatorname{Gr}(2, \mathbb{C}^{5})}{T})\setminus p^{-1}(\operatorname{int}(P_{C} \cup P_{D_{i} })) \big) \cup \big((Y \times D^{4}) \cup p^{-1}(\partial (P_{C} \cup P_{D_{i}}))\big)$ the central chamber of the space $(\sfrac{\operatorname{Gr}(2, \mathbb{C}^{5})}{T})\setminus p^{-1}(\operatorname{int}(P_{C} \cup P_{D_{i} })) \big) \cup \big((Y \times D^{4}) \cup p^{-1}(\partial (P_{C} \cup P_{D_{i}}))\big)$.
      \end{definition}

     In the subsequent lemma, we will show $p^{-1}(P_{C} \cup P_{D_{i}}) \simeq (Y \times D^{4}) \cup p^{-1}(\partial (P_{C} \cup P_{D_{i}}) )$. Subsequently, by using an inductive application of the previous claim, we will show the main theorem of this section.

    \begin{lemma}\label{CentCham}
	Let $P_{C}$ be the central chamber and $P_{D_{i}}$ one of its neighboring chambers. Let $p: \sfrac{\operatorname{Gr}(2,\mathbb{C}^{5})}{T} \longrightarrow P$ be the natural projection. Then, we have the homotopy equivalence $p^{-1}(P_{C} \cup P_{D_{i}}) \simeq (Y \times D^{4}) \cup p^{-1}(\partial (P_{C} \cup P_{D_{i}}) )$.
     \end{lemma}
     \begin{proof}
     We start the proof with the following topological consideration. Let $Y_{D_{i}}$ and $Y_{D_{j}}$ be the attached birational model to two neighboring chambers separated by the wall $\widetilde{W}$. Without loss of the generality, we assume that $D_{j} \subset D_{i}$. Hence, we have the contraction map $\phi: Y_{D_{i}} \longrightarrow Y_{D_{j}}$. It follows that the associated birational model with the interior of the separating wall $\widetilde{W}$ is $Y_{D_{j}}$. This observation holds for lower-dimensional faces in the following sense. Let $W_1, \dots , W_i$ be $n$-dimensional faces in the interior of $P:=\Delta_{2,5}$ with a non-empty intersection of topological boundaries, where $n \leq 3$. Let $F$ be the $(n-1)$-dimensional face shared by $W_1, \dots , W_i$. Let $Y_{D_{1}}, \dots , Y_{D_{i}}$ be the birational models associated with the interior of $W_1, \dots , W_i$, respectively. Then, the attached birational model to the interior of $F$ is $Y_{D_{p}} \in \{ Y_{D_{1}}, \dots , Y_{D_{i}}  \}$, where $\forall \: l \in \{1,\dots, i\} \backslash \{p\}$ there is a contraction map $Y_{D_{l}} \longrightarrow Y_{D_{p}}$.

     Now let $S := P_{D_{i}} \cap P_{C}$. In other words, the topological space $S \cong D^{3}$ is the wall that separates $P_{D_{i}}$ from $P_{C}$. Consider a 4-disk $D_{a}^{4} \subset P_{C}$ such that $S \subset \partial D_{a}^{4}$ and $\partial D_{a}^{4} \cap (\partial P_{C} \setminus S) = \emptyset$. Let $W:= \big(\partial (P_{D_{i}}) \setminus \operatorname{int}(S)\big) \cong D^{3}$ and $D^{4}_{b} := D_{a}^{4} \cup P_{D_{i}}$. In the following, we will show that $p^{-1}(D^{4}_{b} ) \simeq p^{-1}(W)$. Recall that, the face relations of the chamber decomposition induce a regular CW complex on $W$. Consider $D^{3} \subset \partial D^{4}_{b}$ such that $D^{3} \subset \operatorname{int}(P_{C})$. We endow $D^{3}$ with the regular CW complex on $W$. We associate each $n$-cell of $W$ with an $n$-cell of $D^{3}$ for $n \leq 3$. We define the map $f : Y \times D^{3} \longrightarrow p^{-1}(W)$ as follows. Let $e^{n}_{W_{j}}$ be an open cell of $W$ associated with the open cell $e^{n}_{D_{j}^{3}}$ of $D^{3}$, and let $Y_{D_{W_{j}}}$ denote the birational model attached to $e^{n}_{W_{j}}$. Then, we define $f \vert_{e^{n}_{D_{j}^{3} } } := (\phi_{D_{W_{j} }}, \operatorname{id} ) :Y \times e^{n}_{D^{3}_{j} } \longrightarrow Y_{D_{W_{j} } } \times e^{n}_{W_{j} } $, where $\phi_{D_{W_{j} }}: Y \longrightarrow Y_{D_{W_{j} } }$ is the contraction map. We define the map $f$ on each open cell of $D^{3}$, accordingly. From the above topological considerations, it follows that the attached birational model to $\operatorname{int}(S)$ is $Y_{D_{i}}$. Note that, we have $Y_{D_{W_{j} }} \simeq \operatorname{Cyl}(Y \longrightarrow Y_{D_{W_{j}}}) \simeq \operatorname{Cyl}(Y \longrightarrow Y_{D_{i}} ) \cup_{\operatorname{id}} \operatorname{Cyl}(Y_{D_{i}} \longrightarrow Y_{D_{W_{j}} }) $, where $\operatorname{Cyl}$ denotes the mapping cylinder. On the other hand, we have $\operatorname{Cyl}(f) \simeq p^{-1}(D^{4}_{b})$. Consequently, it follows that $p^{-1}(W) \simeq \operatorname{Cyl}(f) \simeq p^{-1}(D^{4}_{b})$. Hence, we can define a homotopy equivalence $h: p^{-1}(P_{C} \cup P_{D_{i}}) \longrightarrow (Y \times D^{4}) \cup p^{-1}(\partial (P_{C} \cup P_{D_{i}}) ) $, that is $h \vert_{p^{-1}}(D^{4}_{b})  : p^{-1}(D^{4}_{b} ) \xrightarrow{\simeq} p^{-1}(W)$, and the identity elsewhere.
     \end{proof}
		
     Let $P_{C}$ and $P_{D}$ denote the central chamber and one of its neighboring chambers. By the above lemma, we have $p^{-1}(P_{C} \cup P_{D_{i}}) \simeq (Y \times D^{4}) \cup p^{-1}(\partial (P_{C} \cup P_{D_{i}}))$. As a result, we arrive at $\sfrac{\operatorname{Gr}(2, \mathbb{C}^{5})}{T} \simeq \big((\sfrac{\operatorname{Gr}(2, \mathbb{C}^{5})}{T})\setminus p^{-1}(\operatorname{int}(P_{C} \cup P_{D_{i} })) \big) \cup \big((Y \times D^{4}) \cup p^{-1}(\partial (P_{C} \cup P_{D_{i}}))\big)$, where the latter topological space is defined in Definition \ref{Replacing}. 

    \begin{theorem}\label{Theo1}
	The topological space $X$, defined in Equation \ref{HomModel}, is homotopy equivalent to the orbit space $\sfrac{\operatorname{Gr}(2,\mathbb{C}^{5})}{T}$. 
    \end{theorem} 
    \begin{proof}
    By Lemma \ref{CentCham}, we can replace all chambers associated with a blow-up of $\mathbb{CP}^{2}$ at 3 or 2 points in general positions with a copy of the central chamber, inductively. Since there is no contraction map in either direction between $\mathbb{CP}^{2} \# \overline{\mathbb{CP}^{2}}$ and $\mathbb{CP}^{1} \times \mathbb{CP}^{1}$, their associated chambers are not adjacent to one another. This can also be easily confirmed from \cite[Figure 2]{bohm2020computing}. In the first step, we left the chambers associated with $ \mathbb{CP}^{2}$ and $\mathbb{CP}^{2} \# \overline{\mathbb{CP}^{2}}$ unchanged.

    Each chamber associated with $\mathbb{CP}^{1} \times \mathbb{CP}^{1}$ has an $O_{i}$ in its topological boundary. Using the same ideas as in the proof of Lemma \ref{CentCham} and the commutative diagram 
      \begin{align*}
    	\xymatrix{Y \ar[d] \ar[dr] &  \\
    		\mathbb{CP}^{1} \times \mathbb{CP}^{1} \ar[r] & \mathbb{CP}_{i}^{1}} 
    \end{align*} 
    implies that we can replace the chambers associated with $\mathbb{CP}^{1} \times \mathbb{CP}^{1}$ with a copy of the central chamber. Recall that the map $f_{i}:Y \longrightarrow \mathbb{CP}^{1}_{i}$ is a distinct conic bundle described in Section \ref{TopConst}. Let  $\mathbb{CP}_{i}^{2} \# \overline{\mathbb{CP}^{2}_{i}}$ and $\mathbb{CP}_{j}^{2} \# \overline{\mathbb{CP}^{2}_{j}}$ be the birational models associated with chambers $P_i$ and $P_{j}$ with octahedrons $O_{i}$ and $O_{j}$ in their topological boundaries, respectively. Let $\mathbb{CP}^{2}$ be the birational model attached to a chamber $P_{k}$ such that $\partial P_{K} \cap \partial P_{j} \neq \emptyset $ and $\partial P_{K} \cap \partial P_{i} \neq \emptyset $. Let $\mathbb{CP}^{1}_{i}$ and $\mathbb{CP}^{1}_{j}$ be the associated varieties with $O_{i}$ and $O_{j}$, respectively. Hence, we have the following commutative diagram
     \begin{align*}
    	\xymatrix{
    		Y \ar[d]_{\phi_{j}} \ar[r]^{\phi_{i} \phantom{---}} & \mathbb{CP}_{i}^{2} \# \overline{\mathbb{CP}^{2}_{i}} \ar[d] \\
    	\mathbb{CP}_{j}^{2} \# \overline{\mathbb{CP}^{2}_{j}} \ar[r] & \mathbb{CP}^{2}.
    	}
    \end{align*}
    We replace the chamber $P_{K}$ with a copy of the central chamber, and this copy is attached to the tetrahedron in the topological boundary of $P_{K}$ using the trivial map $Y \longrightarrow \ast$. Recall that the maps $f_{i}:Y \xrightarrow{\phi_{i}} \mathbb{CP}_{i}^{2} \# \overline{\mathbb{CP}^{2}_{i}} \longrightarrow \mathbb{CP}^{1}_{i}$ and $f_{j}:Y \xrightarrow{\phi_{i}} \mathbb{CP}_{j}^{2} \# \overline{\mathbb{CP}^{2}_{j}} \longrightarrow \mathbb{CP}^{1}_{j}$ are conic bundles. We replace the chambers associated to $\mathbb{CP}_{i}^{2} \# \overline{\mathbb{CP}^{2}_{i}}$ and $\mathbb{CP}_{j}^{2} \# \overline{\mathbb{CP}^{2}_{j}}$ with copies of the central chamber and it follows that these chambers are attached to $\mathbb{CP}^{1}_{i}$ and $\mathbb{CP}^{1}_{j}$ associated with octahedra $\mathcal{O}_{i}$ and $\mathcal{O}_{j}$using the maps $f_{i}$ and $f_{j}$, respectively. Note that the conic bundles $f_{i}$ and $f_{j}$ are different in the sense that the induced homomorphisms $(f_{i})_{\ast}$ and $(f_{j})_{\ast}$ correspond to different rows of the matrix $f_{\ast}$ defined in Equality \ref{bound1}.

    Note that an octahedron $\mathcal{O}_{i}$ lies in the topological boundaries of more than one chamber. Hence, we need to show that for a given octahedron $\mathcal{O}_{i}$, replacing all chambers with $\mathcal{O}_{i}$ in their topological boundary, with a copy of the central chamber does not change the homotopy type of the initial space. Let $\mathbb{CP}^{1}_{i}$ be attached to $O_{i}$ and $Y_D$ be the birational model associated with a chamber with $O_i$ in its topological boundary. It follows that $Y_{D} \in \{ \mathbb{CP}^2 \# \overline{\mathbb{CP}^{2}}, \mathbb{CP}^{1} \times \mathbb{CP}^{1} \}$. Note that, we have the following commutative diagram
    \begin{align}\label{CommDiag2}
    	\xymatrix{Y \ar[d] \ar[r] &\mathbb{CP}^2 \# \overline{\mathbb{CP}^{2}} \ar[d] \\
    		\mathbb{CP}^{1} \times \mathbb{CP}^{1} \ar[r] & \mathbb{CP}_{i}^{1},} 
    \end{align}
    where the composition map $f_{i}:Y \longrightarrow \mathbb{CP}^{1}_{i}$ defines a conic bundle structure on $Y$. The commutativity of Diagram \ref{CommDiag2} ensures that after replacing chambers associated with  $\mathbb{CP}^{2} \# \overline{\mathbb{CP}^{2}}$ and $\mathbb{CP}^{1} \times \mathbb{CP}^{1}$ the resulting space is well-defined and replacing the chambers with $\mathcal{O}_{i}$ in their topological boundaries with a copy of the central chamber does not change the homotopy type of the initial space.
    
    After replacing all chambers with a copy of the central chamber, denote the resulting space by $\widetilde{X}$. The topological space $\widetilde{X}$ has one chamber associated with the degree $5$ del Pezzo surface Y. The chamber is attached to each $\mathbb{CP}^{1}_{i}$ associated with the interior of the octahedron $\mathcal{O}_{i}$ using the conic bundle map $f_{i}$. Furthermore, the chamber is attached to each tetrahedron $T_{i}$ and the topological boundaries of octahedra using the trivial map $Y \longrightarrow \ast$. Note that Lemma \ref{CentCham} implies that after replacing all chambers with a copy of the central chamber the homotopy type of the orbit space $\sfrac{\operatorname{Gr}(2,\mathbb{C}^{5})}{T}$ remains unchanged. From the construction it follows that $X \cong \widetilde{X}$ and $\widetilde{X} \simeq \sfrac{\operatorname{Gr}(2,\mathbb{C}^{5})}{T}$. Hence, we arrive at $X \simeq \sfrac{\operatorname{Gr}(2,\mathbb{C}^{5})}{T}$.
    \end{proof}

    \section{Homotopy Type of the Orbit Space}

    In the previous section, we showed that the orbit space $\sfrac{\operatorname{Gr}(2,\mathbb{C}^{5})}{T}$ is homotopy equivalent to the topological space $X$, defined in Equation \ref{HomModel}. In this section, we aim to study the space $X$. Let $\Sigma A$ denote the unreduced suspension of the topological space $A$. We denote the suspension of a map $f:A \longrightarrow B$ by $\Sigma f: \Sigma A \longrightarrow \Sigma B$. Since the 4-skeleton of $X$ is homeomorphic to $D^{4}$, the topological space $X$ is simply connected and we have $X \simeq \sfrac{X}{X^{4}} \vee D^{4} \simeq \sfrac{X}{X^{4}}$.
    As a result, we can endow the space $\sfrac{X}{X^{4}}$ with the CW complex $\sfrac{X}{X^{4}}=e^{8} \bigcup_{i=1}^{5}e^{6}_{i} \bigcup_{i=1}^{5}e^{5}_{i} \cup e^{0}$.
    Since the spaces $X$ and hence $\sfrac{X}{X^{4}}$ are simply connected,the 6-skeleton of $\sfrac{X}{X^{4}}$ is homotopy equivalent to a Moore space with $H_{5}((\sfrac{X}{X^{4}})^{6}) \cong \mathbb{Z}_{2}$ and hence homotopy equivalent to $\Sigma^{4} \mathbb{RP}^{2}$. Thus, we have 
     $X \simeq	\sfrac{X}{X^{4}} \simeq D^{8} \bigcup_{f} \Sigma^{4} \mathbb{RP}^{2}$,
      where $[f] \in \pi_{7}(\Sigma^{4}\mathbb{RP}^{2})$.
      
      The following lemma is an elementary result about the n-th suspension of a topological space. 
      \begin{lemma}\label{susn}
    	Let $A$ be a topological space. Then $\Sigma^{n}A \cong \sfrac{(D^{n} \times A)}{\sim}$, where for $y, y^{\prime} \in D^{n}$, and $a,a^{\prime} \in A$, we set
    	\begin{align*}
    	(y , a) \sim 
    	\left\{ \begin{array}{rclc}
    	(y^{\prime}, \ast) & \text{if}
    & y,y^{\prime} \in \partial D^{n} &\text{, $y=y^{\prime}$, and for a fixed point $\ast \in A$}\\
	(y^{\prime}, a^{\prime}) & \text{if} & y,y^{\prime} \in \operatorname{int}(D^{n}) &\text{, $y=y^{\prime}$, and $a=a^{\prime}$.}\\
	\end{array}\right.
	\end{align*}
    \end{lemma}
    \begin{proof}
	We show the statement for $n=2$. The proof for the general case goes along the same lines. Let $\Sigma A = (A\times \mathring{I}) \cup \{\ast\} \cup \{\ast\}$, then $\Sigma^{2} A = \sfrac{\big( (X \times \mathring{I}) \cup \{ \ast \} \cup \{ \ast \} \big) \times I}{\sim^{\prime}} \cong (A \times \mathring{I} \times \mathring{I}) \cup (\{\ast\} \times I) \cup (\{\ast\} \times I)$. Hence, we arrive at $\Sigma^{2}A \cong (A \times \mathring{D}^{2}) \cup (\partial D^{2})$.
    \end{proof}
    \begin{theorem}\label{Theo2}
	Let 	$X \simeq D^{8} \bigcup_{f} \Sigma^{4} \mathbb{RP}^{2}$. Then, the equivalence class $[f] \in \pi_{7}(\Sigma^{4} \mathbb{RP}^{2} )$ is a generator. Furthermore, we have $D^{8} \bigcup_{f} \Sigma^{4} \mathbb{RP}^{2} \not \simeq \Sigma^{4}(D^{4} \bigcup_{g} \mathbb{RP}^{2}) $ for all $[g] \in \pi_{3}(\mathbb{RP}^{2})$.
    \end{theorem}
    \begin{proof}
	Consider the double covering $c:S^{2} \longrightarrow \mathbb{RP}^{2} $. The induced homomorphism $c_{3}: \pi_{3}(S^{2}) \longrightarrow \pi_{3}(\mathbb{RP}^{2})$ is an isomorphism. Choose any attaching map $g^{\prime}:S^{3} \longrightarrow \mathbb{RP}^{2}$. Since $c_{3}$ is an isomorphism, we can find $g \simeq g^{\prime}$ such that $g: S^{3} \longrightarrow S^{2} \xrightarrow{\phantom{-}c \phantom{-}} \mathbb{RP}^{2}$. It follows that $\Sigma^{4}g \simeq \Sigma^{4}g^{\prime}$ . In \cite{wu2003homotopy}, Wu shows that $\pi_{7}(\Sigma^{4} \mathbb{RP}^{2}) \cong \mathbb{Z}_{4}$. Consider the homomorphism $(\Sigma^{4}c)_{7}: \pi_{7}(S^{6}) (\cong \mathbb{Z}_{2}) \longrightarrow \pi_{7}(\Sigma^{4} \mathbb{RP}^{2}) (\cong \mathbb{Z}_{4})$. Since $\operatorname{Hom}(\mathbb{Z}_{2},\mathbb{Z}_{4}) \cong \mathbb{Z}_{2}$ an element of the homotopy group $\pi_{7}(S^{6})$ under the induced homomorphism $(\Sigma^{4}c)_{7}$ is mapped either to $0$ or $2$ in $\mathbb{Z}_{4}$. It follows that the homotopy class $(\Sigma^{4}c)_{7}([S^{3} \longrightarrow S^{2}]) \in \pi_{7}(\Sigma^{4} \mathbb{RP}^{2})$ is either 0 or 2 and not a generator. It follows that the homotopy class of the fourth suspension of the attaching map $g'$, namely $[\Sigma^{4}g'] \in \pi_{7}(\Sigma^{4} \mathbb{RP}^{2})$, can not be a generator. It remains to show that the attaching map $f:S^{7} \longrightarrow \Sigma^{4}\mathbb{RP}^{2}$ is a generator.

    Let $X^{6}$ be the 6-skeleton of $X$. From the above considerations, it follows that $\pi_{7}(X^{6}) \cong \mathbb{Z}_{4}$. We define the topological space $\sfrac{X}{\sim}$ as follows. We set $(x, y) \sim (x^{\prime}, \ast)$ for all $x,x^{\prime} \in \operatorname{int}(O_{i})$, and every $y \in \mathbb{CP}^{1}$, and a fixed point $\ast \in \mathbb{CP}^{1}$, if $x=x^{\prime}$, and perform no non-trivial identification elsewhere. In other words, we collapse each $\mathbb{CP}^1$ attached to a point $x \in \operatorname{int}(\mathcal{O})_{i}$ to a point. As a result of Lemma \ref{susn}, we have $\sfrac{X^{6}}{\sim} \cong (\Sigma^{4} Y)^{6}$, where $Y$ is the blow-up of $\mathbb{CP}^{2}$ at four points in general position. The composition of the quotient maps $X^{6} \longrightarrow \sfrac{X^{6}}{\sim}$ and $\sfrac{X^6}{\sim} \big( \cong (\Sigma^{4} Y)^{6} \big) \longrightarrow \sfrac{(\Sigma^{4} Y)^{6}}{((\Sigma^{4} Y)^6 -e^{6}_{i} )} \big(\cong S^{6}_{i} \big)$, denoted by $p$, induces the homomorphism $p_{7}: \pi_{7}(X^{6})(\cong \mathbb{Z}_{4}) \longrightarrow  \pi_{7}(S^{6}_{i})(\cong \mathbb{Z}_{2})$. Note that $\operatorname{Hom}(\mathbb{Z}_{4},\mathbb{Z}_{2}) \cong \mathbb{Z}_{2}$. In the following, we show that the induced homomorphism $p_{7}$ corresponds to the non-trivial element of $\operatorname{Hom}(\mathbb{Z}_{4},\mathbb{Z}_{2})$. Consider the following composition of maps
    \begin{align*}
	\alpha:S^{7} \xrightarrow{f} X^{6} \rightarrow \sfrac{X^6}{\sim} \big( \cong (\Sigma^{4} Y)^{6} \big) \rightarrow \sfrac{(\Sigma^{4} Y)^{6}}{((\Sigma^{4} Y)^6 -e^{6}_{i} )} \big( \cong S^{6}_{i}\big),
    \end{align*}	
    where $f$ is the attaching map of $X$. Since the topological space $Y$ is homeomorphic to the blow-up of $\mathbb{CP}^2$ on four points in general position, the map $\alpha: S^{7} \longrightarrow S^{6}_{i}$ is homotopic to $\Sigma^{4}(h)$ or $\Sigma^{4}(-h)$, where $h:S^{3} \longrightarrow S^{2}$ is the Hopf fibration. It is well known that $\Sigma^{4}(h)$ and $\Sigma^{4}(-h)$ are not null-homotopic. It follows that the homomorphism $p_{7}$ is the non-trivial element of $\operatorname{Hom}(\mathbb{Z}_{4},\mathbb{Z}_{2})$, and hence $p_{7}([f])$ is the non-trivial element of $\pi_{7}(S^{6}_{i})\cong \mathbb{Z}_{2}$. Hence, the homotopy class $[f] \in \pi_{7}(X^{6})$ is a generator. 
    \end{proof}

\end{document}